\documentclass[a4paper,abstracton]{scrartcl}
\usepackage{amssymb}
\usepackage{amsmath}
\usepackage{amsthm}
\usepackage{graphicx}
\usepackage{fancyhdr}              
\usepackage{lastpage}              
                                   
\usepackage[shortlabels]{enumitem}
\usepackage{esint}

\usepackage [english]{babel}

\usepackage{setspace}

\newcommand{\setR}{\mathbb{R}}
\newcommand{\setC}{\mathbb{C}}
\newcommand{\setZ}{\mathbb{Z}}
\newcommand{\setQ}{\mathbb{Q}}

\newcommand{\setP}{\mathbb{P}}

\newtheorem{definition}{Definition}
\newtheorem{theorem}[definition]{Theorem}
\newtheorem{lemma}[definition]{Lemma}

\newcommand{\set}[1]{\left\lbrace #1 \right\rbrace}
\newcommand{\card}[1]{\#\set{#1}}

\newcommand{\divides}{\mid}

\DeclareMathOperator{\dfunc}{d}

\DeclareMathOperator{\Img}{Im}
\DeclareMathOperator{\Ker}{Ker}

\begin{document}
\title{Counting points on bilinear and trilinear hypersurfaces}
\author{Thomas Reuss\\ 
		Mathematical Institute, University of Oxford\\
		reuss@maths.ox.ac.uk}
\maketitle

\begin{abstract}
Consider an irreducible bilinear form $f(x_1,x_2;y_1,y_2)$ with integer coefficients. We derive an upper bound for the number of integer points $(\mathbf{x},\mathbf{y})\in\mathbb{P}^1\times\mathbb{P}^1$ inside a box satisfying the equation $f=0$. Our bound seems to be the best possible bound and the main term decreases with a larger determinant of the form $f$. We further discuss the case when $f(x_1,x_2;y_1,y_2;z_1,z_2)$ is an irreducible non-singular trilinear form defined on $\mathbb{P}^1\times \mathbb{P}^1\times\mathbb{P}^1$, with integer coefficients. In this case, we examine the singularity and reducibility conditions of $f$. To do this, we employ the Cayley hyperdeterminant $D$ associated to $f$. We then derive an upper bound for the number of integer points in boxes on such trilinear forms. The main term of the estimate improves with larger $D$. Our methods are based on elementary lattice results.
\end{abstract}


\section{Bilinear Forms}
We note the following fact about bilinear forms:
\begin{lemma}
Let $A=(a_{ij})$ be a $2\times 2$ matrix with associated bilinear form
$f(\mathbf{x},\mathbf{y})=\mathbf{x}^TA\mathbf{y}$. Then $f$ is irreducible if and only if
$\det(A)\neq 0$. If $f$ is reducible then it factorizes into a product of two linear factors.
\end{lemma}
\begin{proof}
This is a well known result and can quickly be verified by direct calculation.
\end{proof}

We will prove the following theorem:
\begin{theorem} \label{thm:thmbilinear}
Let $A=(a_{ij})$ be a $2\times 2$ matrix with associated bilinear form 
\[
	f(\mathbf{x},\mathbf{y})=\mathbf{x}^TA\mathbf{y}
\]
and determinant $\Delta:=\det(A)$. Suppose that $\gcd(a_{11},a_{21},a_{12},a_{22})=1$.
We assume that $f$ is irreducible over $\setZ$ so that $\Delta\neq 0$.
Let $X_1,X_2,Y_1,Y_2$ be real numbers $\geq 1$ and define
\[
	\mathcal{N}(\mathbf{X},\mathbf{Y})=\card{(\mathbf{x},\mathbf{y}): |x_i|\leq X_i,
	|y_i|\leq Y_i, (x_1,x_2)=(y_1,y_2)=1, f(\mathbf{x},\mathbf{y})=0}.
\]
Then
\[
	\mathcal{N}(\mathbf{X},\mathbf{Y})\ll
	\min\left\{X_1X_2,Y_1Y_2,\dfunc(\Delta)\left(\sqrt{\frac{X_1X_2Y_1Y_2}{|\Delta|}}+1\right)\right\}.
\]
\end{theorem}
\begin{proof}
We want to count solutions to the equation
\[
	f(\mathbf{x},\mathbf{y})= x_1(a_{11}y_1+a_{12}y_2)+x_2(a_{21}y_1+a_{22}y_2)=0,
\]
say. Let us write this equation as
\[
	f(\mathbf{x},\mathbf{y})= x_1L_1(\mathbf{y})+x_2L_2(\mathbf{y})=0.
\]
First let us consider the case when $L_1(\mathbf{y})=0$. Note that $\mathbf{y}$ is a
primitive vector and hence in particular, it is non-zero.
If furthermore, $L_2(\mathbf{y})=0$ then we have a non-zero solution to the equation
$A\mathbf{y}=0$ which is impossible since $\Delta\neq 0$. Thus, we have $x_2=0$ which implies that
$x_1=\pm 1$ since $\mathbf{x}$ is a primitive vector. There are at most 4 choices for $\mathbf{y}$ such that  $L_1(\mathbf{y})=0$. This is easy to see after we divide $a_{11}$ and $a_{12}$ by $(a_{11},a_{12})$
in the equation $a_{11}y_1+a_{12}y_2=0$.
Thus, the case $L_1(\mathbf{y})=0$ contributes $O(1)$ to $\mathcal{N}(\mathbf{X},\mathbf{Y})$.
The case $L_2(\mathbf{y})=0$ is analogous. Now, if $x_1=0$ then $x_2=\pm 1$ which reduces to the
case $L_2(\mathbf{y})=0$. Thus, we may assume that
\[
	x_1x_2L_1(\mathbf{y})L_1(\mathbf{y})\neq 0.
\]
Since $x_1$ and $x_2$ are coprime, we can therefore deduce the existence of a non-zero integer $q$ such that
\begin{eqnarray*}
	-q x_1 &=& a_{21}y_1+a_{22}y_2,\\
	 q x_2 &=& a_{11}y_1+a_{12}y_2.
\end{eqnarray*}
If we set 
\[
	P=\left(\begin{array}{cc}
	0 & 1 \\ -1 & 0 \\	
	\end{array} \right),
\]
then we have $qP\mathbf{x}=A\mathbf{y}$.
Interchanging the roles of $\mathbf{x}$ and $\mathbf{y}$, we can similarly deduce that there
is a non-zero integer $q'$ such that $q'P\mathbf{y}=A^T\mathbf{x}$.
Combining the two equations, we get
\[
	qq'\mathbf{x}=-\Delta\mathbf{x}.
\]
The vector $\mathbf{x}$ is primitive and therefore non-zero. This allows us to conclude that
\[
	qq'=-\Delta.
\]
Since $\Delta$ is non-zero, there are $\dfunc(\Delta)$ choices for $q$.
We now fix one such choice and count the respective contribution to 
$\mathcal{N}(\mathbf{X},\mathbf{Y})$.
We first note that $q\divides A\mathbf{y}$. This will give us a lattice condition on $\mathbf{y}$, which is described in the following lemma:

\begin{lemma} \label{lem:lattice}
Fix an integer $m\geq 2$. Let $M=(m_{ij})$ be an $m\times 2$ matrix. Let $q$ be a non-zero integer such that $q$ divides all of the $m(m-1)/2$ minors of size $2\times 2$ of $M$. Furthermore, assume that there exists no prime $p\divides q$ which divides all the entries of $M$. Then
\[
	\Lambda_q=\{\mathbf{v}\in\setZ^2: q\divides M\mathbf{v}\}
\]
is a lattice of determinant $q$.
\end{lemma}
\begin{proof} (of Lemma) It is clear that $\Lambda_q$ is an integer lattice. We proceed to calculate
its determinant. We decompose $q=\prod p_i^{e_i}$ into its prime powers and consider the (additive) group homomorphism
\[
	\phi_i:\setZ^2\rightarrow \setZ^m/p_i^{e_i}\setZ^m
\]
given by $\phi_i(\mathbf{v})=M\mathbf{v}\ (\text{mod }p_i^{e_i})$. By assumption of the lemma,
there exists an element of $M$ which is not divisible by $p_i$. We assume without loss of generality that $p_i\nmid m_{11}$. The other cases are analogous. 
Let $m_{11}^{-1}$ be the multiplicative inverse of $m_{11}$ modulo $p_i^{e_i}$.
Let $\mathbf{w}\equiv m_{11}^{-1}(m_{11},m_{21},\ldots,m_{m1})^T\ (\text{mod }p_i^{e_i})$. We claim that
$\Img(\phi_i)$ is cyclic with generator $\mathbf{w}$. First note that for any 
$\lambda\in\setZ/p_i^{e_i}\setZ$,
$\phi_i(\lambda m_{11}^{-1},0)\equiv \lambda \mathbf{w}\ (\text{mod }p_i^{e_i})$.
Furthermore, assume $\mathbf{u}=M\mathbf{v}$. Then for $j=2,\ldots,m$:
\[
	u_1m_{j1}m_{11}^{-1} = m_{j1}v_1+m_{12}m_{j1}m_{11}^{-1}v_2\equiv m_{j1}v_1+m_{j2}v_2=u_j,
\]
where the last equality follows from the fact that $p_i^{e_i}$
divides the minor $m_{11}m_{i2}-m_{12}m_{i1}$.
Hence, we have indeed that $\mathbf{u}=u_1\mathbf{w}$, and $\mathbf{w}$ generates the image of $\phi_i$. We now consider the group homomorphism
\[
	\phi:\setZ^2\rightarrow \setZ^m/q\setZ^m
\]
given by $\phi(\mathbf{v})=M\mathbf{v}\ (\text{mod }q)$. By the Chinese Remainder Theorem and by what we
just showed, we have that $|\Img(\phi)|=q$. We note that $\Ker(\phi)=\Lambda_q$ and it follows by the
first Isomorphism Theorem for groups that  

\[
	\det(\Lambda_q)=[\setZ^2:\Lambda_q]=|\text{Im}(\phi)|=q,
\]
which proves the lemma.
\end{proof}

We now recall that $q\divides A\mathbf{y}$. Thus, Lemma \ref{lem:lattice} shows that $\mathbf{y}\in\Lambda_y$, where $\Lambda_y$ is a lattice of determinant $|q|$.
The bounds $|y_1|\leq Y_1$ and $|y_2|\leq Y_2$ let us deduce that $\mathbf{y}$ is inside the ellipse
defined by 
\[
	E_y=\left\{\mathbf{y}\in\setR^2:
	    \left(\frac{y_1}{ \sqrt{2} Y_1}\right)^2+\left(\frac{y_2}{\sqrt{2} Y_2}\right)^2\leq 1\right\}.
\]
The area of this ellipse is $A(E_y)=2\pi Y_1 Y_2$.
Similarly, $\mathbf{x}\in E_x\cap \Lambda_x$, where $E_x$ is an ellipse of area $A(E_x)=2\pi X_1X_2$,
and $\Lambda_x$ is an integer lattice of determinant $|q'|$.
Thus, we have reduced the problem to one where we need to count primitive lattice points inside an ellipse. We will employ Lemma 2 of Heath-Brown \cite{RHBapprox}, which we state here for convenience:

\begin{lemma}\label{lem:lemEllipse}
Let $\Lambda\subseteq\setR^2$ be a lattice of determinant $\det(\Lambda)$.
Let $E\subseteq\setR^2$ be an ellipse, centered at the origin, together with its interior,
and let $A(E)$ be the area of $E$.
Then there is a positive number $\alpha=\alpha(\Lambda,E)$ and a basis $\mathbf{b}_1$,
$\mathbf{b}_2$ of $\Lambda$ such that $\lambda_1 \mathbf{b}_1+\lambda_2 \mathbf{b}_2\in E$
implies $|\lambda_1|\leq\alpha$ and $|\lambda_2|\leq A(E)/(\alpha\det(\Lambda).$
Furthermore, the number of primitive lattice points in $\Lambda$ contained in $E$ is at most
\[
	4\left(\frac{A(E)}{\det(\Lambda)}+1\right).
\]
\end{lemma}

Thus, 
\[
	\mathcal{N}(\mathbf{X},\mathbf{Y})\ll \sum_{qq'=-\Delta}
	\min\left\{\frac{X_1X_2}{|q'|},
			   \frac{Y_1Y_2}{|q|}\right\}+\dfunc(\Delta)
\]
Since $qq'=-\Delta$, the worst case for $q$ is when $q^2=|\Delta| Y_1Y_2/(X_1X_2)$.
Thus, we get the bound
\[
	\mathcal{N}(\mathbf{X},\mathbf{Y})\ll
	\dfunc(\Delta)\left(\sqrt{\frac{X_1X_2Y_1Y_2}{|\Delta|}}+1\right).
\]
The bound 
\[
	\mathcal{N}(\mathbf{X},\mathbf{Y})\ll \min(X_1X_2,Y_1,Y_2)
\]
can be deduced as follows.
There are $O(X_1X_2)$ choices for $\mathbf{x}$. Fix one such $\mathbf{x}$. The equation
$f(\mathbf{x},\mathbf{y})=0$ cannot be zero identically because $\mathbf{x}$ is a primitive vector and
$\Delta\neq 0$. Thus, there are $O(1)$ choices for $\mathbf{y}$. A similar argument yields the bound $O(Y_1Y_2)$. This proves the theorem.
\end{proof}\clearpage
\section{Trilinear Forms}
Let $A=(a_{ijk})$ be a $2\times 2\times 2$ hypermatrix. We associate with $A$ its hyperdeterminant
\[
\begin{split}
	D :=\ 
		& a_{122}^2 a_{211}^2 + a_{111}^{2}a_{222}^{2} +a_{212}^{2}a_{121}^{2}
		 +a_{112}^{2}a_{221}^{2} \\
		-&2 a_{{111}}a_{{122}}a_{{211}}a_{{222}} -2 a_{{211}}a_{{122}}a_{{112}}a_{{221}}
		 -2 a_{{211}}a_{{122}}a_{{212}}a_{{121}} -2 a_{{222}}a_{{111}}a_{{212}}a_{{121}} \\
		-&2 a_{{222}}a_{{111}}a_{{112}}a_{{221}} -2 a_{{112}}a_{{121}}a_{{212}}a_{{221}}
		 +4 a_{{112}}a_{{121}}a_{{211}}a_{{222}} +4 a_{{111}}a_{{122}}a_{{212}}a_{{221}}.
\end{split}
\]
We will give a brief outline of some properties of $D$ below. For some further details 
on hyperdeterminants, the reader may consult  the text by Gel'fand, Kapranov and Zelevinsky \cite{GelfandKapranovZelevinsky}.
We define a trilinear form $f=f_A:\setR^3\rightarrow\setR$ by
\[
	f(\mathbf{x},\mathbf{y},\mathbf{z})=\mathbf{x}^T (A\mathbf{z})\mathbf{y}.
\]
Here, $A\mathbf{z}$ denotes the standard hypermatrix-vector multiplication 
along the third dimension of $A$. In particular,
\[
	A\mathbf{z}=M_{xy}(\mathbf{z}):=
	z_{{1}} \left( 
		\begin {array}{cc} 
			a_{{111}}&a_{{121}} \\
			a_{{211}}&a_{{221}}
		\end {array} \right) 
	+z_{{2}} \left( 
		\begin {array}{cc} 
			a_{{112}}&a_{{122}}\\ 
			a_{{212}}&a_{{222}}
		\end {array} 
	\right)
\]
is an ordinary 2x2 matrix depending only on $\mathbf{z}$. We similarly define 2x2 matrixes 
$M_{yz}(\mathbf{x})$ and $M_{zx}(\mathbf{y})$ by
\[
	M_{yz}(\mathbf{x}):=
	x_{{1}} \left( 
		\begin {array}{cc} 
			a_{{122}}&a_{{121}} \\
			a_{{112}}&a_{{111}}
		\end {array} \right) 
	+x_{{2}} \left( 
		\begin {array}{cc} 
			a_{{222}}&a_{{221}}\\ 
			a_{{212}}&a_{{211}}
		\end {array} 
	\right)
\]
and
\[
	M_{xz}(\mathbf{y}):=
	y_{{1}} \left( 
		\begin {array}{cc} 
			a_{{111}}&a_{{112}} \\
			a_{{211}}&a_{{212}}
		\end {array} \right) 
	+y_{{2}} \left( 
		\begin {array}{cc} 
			a_{{121}}&a_{{122}}\\ 
			a_{{221}}&a_{{222}}
		\end {array} 
	\right),
\] 
so that
\[
	f(\mathbf{x},\mathbf{y},\mathbf{z})=\mathbf{x}^T M_{xy}(\mathbf{z})\mathbf{y}
	=\mathbf{y}^T M_{yz}(\mathbf{x})\mathbf{z}=\mathbf{z}^T M_{zx}(\mathbf{y})\mathbf{x}.
\]
By taking transposes on the equation $f(\mathbf{x},\mathbf{y},\mathbf{z})=0$ it makes sense to
define
\[
	M_{yx}(\mathbf{z}):=M_{xy}(\mathbf{z})^T,\quad
	M_{zy}(\mathbf{x}):=M_{yz}(\mathbf{x})^T,\quad
	M_{xz}(\mathbf{y}):=M_{zx}(\mathbf{y})^T.
\]
Furthermore, let
\begin{eqnarray*}
	\Delta_{xy}=\Delta_{xy}(\mathbf{z}) &:=& \det(M_{xy}(\mathbf{z})),\\
	\Delta_{yz}=\Delta_{yz}(\mathbf{x}) &:=& \det(M_{yz}(\mathbf{x})),\\
	\Delta_{zx}=\Delta_{zx}(\mathbf{y}) &:=& \det(M_{zx}(\mathbf{y})).
\end{eqnarray*}
Then 

\begin{align*}
	\Delta_{xy}(\mathbf{z})&=
	\left( a_{{111}}a_{{221}}-a_{{121}}a_{{211}} \right) z_{{1}}^{2}+
 \left( a_{{111}}a_{{222}}+a_{{112}}a_{{221}} -a_{{212}}a_{{121}}-a_{{
211}}a_{{122}} \right) z_{{2}}z_{{1}}\\&+ \left( a_{{112}}a_{{222}}-a_{{
122}}a_{{212}} \right) z_{{2}}^{2},\\
\Delta_{yz}(\mathbf{x})&=
	\left( a_{{111}}a_{{122}}-a_{{112}}a_{{121}} \right) x_{{1}}^{2}+
 \left( a_{{111}}a_{{222}}+a_{{211}}a_{{122}}-a_{{112}}a_{{221}}-a_{{
212}}a_{{121}} \right) x_{{2}}x_{{1}}\\&+ \left( a_{{211}}a_{{222}}-a_{{
212}}a_{{221}} \right) x_{{2}}^{2},\\
\Delta_{zx}(\mathbf{y})&=
	\left( a_{{111}}a_{{212}}-a_{{112}}a_{{211}} \right) y_{{1}}^{2}+
 \left( a_{{111}}a_{{222}}+a_{{212}}a_{{121}}-a_{{112}}a_{{221}}-a_{{
211}}a_{{122}} \right) y_{{2}}y_{{1}}\\&+ \left( a_{{121}}a_{{222}}-a_{{
122}}a_{{221}} \right) y_{{2}}^{2}.
\end{align*}

\begin{lemma}
	The discriminants of the quadratic forms $\Delta_{xy}$, 
	$\Delta_{yz}$ and $\Delta_{zx}$ are all equal to $D$.
\end{lemma}
\begin{proof}
This can be verified by direct calculation. It suffices to show that
\begin{align*}
	D &=\left( a_{{111}}a_{{222}}+a_{{112}}a_{{221}} -a_{{212}}a_{{121}}-a_{{
211}}a_{{122}} \right)^2\\
	&\qquad\qquad -4\left( a_{{111}}a_{{221}}-a_{{121}}a_{{211}} \right)
\left( a_{{112}}a_{{222}}-a_{{
122}}a_{{212}} \right)\\
      &=\left( a_{{111}}a_{{222}}+a_{{211}}a_{{122}}-a_{{112}}a_{{221}}-a_{{
212}}a_{{121}} \right) ^{2}\\
	&\qquad\qquad -4\, \left( a_{{111}}a_{{122}}-a_{{112}}a_{
{121}} \right)  \left( a_{{211}}a_{{222}}-a_{{212}}a_{{221}} \right) \\
	&=\left( a_{{111}}a_{{222}}+a_{{212}}a_{{121}}-a_{{112}}a_{{221}}-a_{{
211}}a_{{122}} \right) ^{2} \\
	&\qquad\qquad -4\, \left( a_{{111}}a_{{212}}-a_{{112}}a_{
{211}} \right)  \left( a_{{121}}a_{{222}}-a_{{122}}a_{{221}} \right) .
\end{align*}
\end{proof}

\begin{lemma}\label{lem:DQ0}
	Assume that 
	\[
		Q(\mathbf{x})=ax_1^2+bx_1x_2+cx_2^2\in\setZ[\mathbf{x}]
	\]
	is a binary quadratic form with discriminant $D(Q)=b^2-4ac=0$. Then, $D(Q)=0$ if and only if
	$Q(\mathbf{x})=C L(\mathbf{x})^2$ for some (possibly zero) $C\in\setQ$, 
	and a linear form $L(\mathbf{x})\in\setZ[\mathbf{x}]$ with integer coefficients.
\end{lemma}

\begin{lemma}\label{lem:factorization}\ 
\begin{enumerate}[i)]
	\item Assume that $\Delta_{xy}$ vanishes identically. Then $\Delta_{yz}$ or $\Delta_{zx}$ 
	also vanishes identically.
	\item If $\Delta_{xy}$ and $\Delta_{yz}$ vanish identically then $f$ factorizes over $\setQ$ as
		\begin{equation}\label{eq:eqfsplitstoLB}
			f(\mathbf{x},\mathbf{y},\mathbf{z}) = L(\mathbf{y})B(\mathbf{x},\mathbf{z}),
		\end{equation}
	where $L$ is a linear form,  and $B$ is a bilinear form.
	Furthermore, $\Delta_{zx}(\mathbf{y})=\det(B)L(\mathbf{y})^2$,
	where $\det(B)$ is the determinant of the matrix associated to $B$.
	\item If $f$ has a linear factor $L(\mathbf{y})$ over $\setZ$ then $\Delta_{xy}$ and 
	$\Delta_{yz}$ both vanish identically.
	\item $f$ splits in three linear factors over $\setZ$ if and only if $\Delta_{xy}$, 
	$\Delta_{yz}$ and $\Delta_{zx}$ all vanish identically.	
\end{enumerate}
\end{lemma}
\begin{proof}
We prove the first two claims when $a_{111}a_{222}\neq 0$. Assume that $\Delta_{xy}$ vanishes
identically. Then   all the coefficients of $\Delta_{xy}$ vanish, that is:
\begin{align}
	a_{{111}}a_{{221}}-a_{{121}}a_{{211}} &=0\nonumber\\
	a_{{112}}a_{{222}}-a_{{122}}a_{{212}} &=0\label{eq:eqDeltavan1}\\
	 a_{{111}}a_{{222}}+a_{{112}}a_{{221}} -a_{{212}}a_{{121}}-a_{{211}}a_{{122}} &=0\nonumber.
\end{align}
Using just the first two equations, we get
\[
	\Delta_{yz}(\mathbf{x}) =
	{\frac { \left( a_{{111}}a_{{222}}-a_{{
212}}a_{{121}} \right) \left( x_{{1}}a_{{122}}+x_{{2}}a_{{222}} \right)  \left( x_{{
1}}a_{{111}}+x_{{2}}a_{{211}} \right) }{a_{{111}}a_{{222}}}}
\] 
and
\[
	\Delta_{zx}(\mathbf{y}) =
	{\frac { \left( a_{{111}}a_{{222}}-a_{{211}}a_{{122}} \right)\left( y_{{1}}a_{{212}}+y_{{2}}a_{{222}} \right)    \left( y_{{1}}a_{{111}}+y_
{{2}}a_{{121}} \right) }{a_{{111}}a_{{222}}}}.
\]
If we substitute the first two equations of \eqref{eq:eqDeltavan1} into the third one, we obtain
that
\[
	{\frac { \left( a_{{111}}a_{{222}}-a_{{211}}a_{{122}} \right) 
 \left( a_{{111}}a_{{222}}-a_{{212}}a_{{121}} \right) }{a_{{111}}a_{{
222}}}}=0.
\]
If $a_{{111}}a_{{222}}-a_{{211}}a_{{122}}=0$ then $\Delta_{zx}(\mathbf{y})$ vanishes identically
and
\[
	\Delta_{yz}(\mathbf{x})={\frac { 
 \left( a_{{212}}a_{{121}}-a_{{211}}a_{{122}} \right) }{a_{{222}}a_{{
122}}}} \left( x_{{1}}a_{{122}}+x_{{2}}a_{{222}} \right) ^{2}.
\]
Furthermore, $f(\mathbf{x},\mathbf{y},\mathbf{z})$ factorizes as follows:
\[
{\frac { \left( x_{{1}}a_{{122}}+x_{{2}}a_{{222}} \right)  \left( a_{{
122}}a_{{212}}z_{{2}}y_{{1}}+a_{{122}}y_{{1}}z_{{1}}a_{{211}}+a_{{122}
}y_{{2}}z_{{2}}a_{{222}}+y_{{2}}z_{{1}}a_{{121}}a_{{222}} \right) }{a_
{{222}}a_{{122}}}}.
\]
This proves i) and ii) in the case under consideration. In the case $a_{{111}}a_{{222}}-a_{{212}}a_{{121}}=0$, we note that $\Delta_{yz}(\mathbf{x})$ vanishes identically
and the calculations are similar. It is not difficult to verify i) and ii) if
$a_{111}a_{222}=0$.\bigskip\\
We now prove iii). If
\[
	 f(\mathbf{x},\mathbf{y},\mathbf{z})=\left( b_{{1}}y_{{1}}+b_{{2}}y_{{2}} \right)  \left( c_{{1}}x_{{1}}z_
{{1}}+c_{{2}}x_{{1}}z_{{2}}+c_{{3}}x_{{2}}z_{{1}}+c_{{4}}x_{{2}}z_{{2}
} \right)
\]
then
\begin{align*}
	a_{111} = b_1c_1,\ 
	&a_{112} = b_1c_2,\ 
	a_{121} = b_2c_1,\ 
	a_{122} = b_2c_2,\\
	a_{211} = b_1c_3,\ 
	&a_{221} = b_2c_3,\ 
	a_{212} = b_1c_4,\ 
	a_{222} = b_2c_4,\ 
\end{align*}
and it can easily be verified by direct calculation that $\Delta_{xy}$ and 
$\Delta_{yz}$ both vanish identically.
The statement iv) follows directly from ii) and iii).
\end{proof}

\begin{lemma}
	The following are equivalent
	\begin{enumerate}[i)]
		\item $D=0$.\label{it:item1}
		\item There exists a non-trivial point in $(\setP^1)^3$ at which all partial derivatives
			  of $f$ vanish.\label{it:item2}
		\item $f$ is singular in $(\setP^1)^3$. \label{it:item3}
	\end{enumerate}
\end{lemma}
\begin{proof}\ 
\begin{itemize}
	\item[\ref{it:item2} $\Rightarrow$ \ref{it:item1}]
	 Assume for a contradiction that all partial derivatives of $f$ vanish at a point
	 in $(\setP^1)^3$ and that $D\neq 0$.
	 Thus, by Lemma \ref{lem:DQ0}, there are precisely two distinct solutions 
	 $\mathbf{z_1}, \mathbf{z_2}\in\setP$ to the equation $\Delta_{xy}(\mathbf{z})=0$.
	 If $M_{xy}(\mathbf{z_1})$ is identically zero, then
	 \[
	 	M_{xz}(\mathbf{y})\mathbf{z_1}=M_{xy}(\mathbf{z_1})\mathbf{y}=\mathbf{0},
	 \]
	 for all $\mathbf{y}\in\setP$. This means that for all $\mathbf{y}\in\setP$, there is a non-zero
	 vector in the kernel of $M_{xz}(\mathbf{y})$. Thus, $\Delta_{xz}(\mathbf{y})$ vanishes identically
	 and therefore $D=0$, which contradicts our assumption. Therefore, $M_{xy}(\mathbf{z_1})$ cannot
	 vanish identically and there must be $\mathbf{y_2}\in\setP$ such that 
	 $M_{xy}(\mathbf{z_1})\mathbf{y_2}=0$. We note that $\mathbf{y_2}$ is unique (up to a scalar
	 multiple), since otherwise $M_{xy}(\mathbf{z_1})$ would vanish identically.
	 Similarly, there exists a unique $\mathbf{y_1}\in\setP$ such that
	  $M_{xy}(\mathbf{z_2})\mathbf{y_1}=0$. If $\mathbf{y_1}=\mathbf{y_2}$ then
	  \[
	  	M_{xz}(\mathbf{y_1})\mathbf{z_1}=M_{xz}(\mathbf{y_1})\mathbf{z_2}.
	  \]
	  But similarly to the above argument, there must be a unique $\mathbf{z}\in\setP$ such that
	  $M_{xz}(\mathbf{y_1})\mathbf{z}=0$. But, by assumption we have that 
	  $\mathbf{z_1}\neq\mathbf{z_2}$. Therefore, $\mathbf{y_1}\neq\mathbf{y_2}$.
	  Similarly, we can find $\mathbf{x_1},\mathbf{x_2}\in\setP$ with $\mathbf{x_1}\neq\mathbf{x_2}$
	  such that $M_{yx}(\mathbf{z_1})\mathbf{x_2}=0$ and $M_{yx}(\mathbf{z_2})\mathbf{x_1}=0$.
	  We can now write an arbitrary $(\mathbf{x},\mathbf{y},\mathbf{z})\in(\setP^1)^3$
	  as 
	  \[
	  	(a_1\mathbf{x_1}+a_2\mathbf{x_2}, b_1\mathbf{y_1}+
	  b_2\mathbf{y_2}, c_1\mathbf{z_1}+c_2\mathbf{z_2}),
	  \] 
	  say. This is possible since $\mathbf{x_1},\mathbf{x_2}$ and $\mathbf{y_1},\mathbf{y_2}$
	  and $\mathbf{z_1},\mathbf{z_2}$ are all bases for $\setP^1$. We can then see that
	  \[
	  	f(\mathbf{x},\mathbf{y},\mathbf{z})=a_1b_1c_1 f(\mathbf{x_1},\mathbf{y_1},\mathbf{z_1})
	  	+a_2b_2c_2 f(\mathbf{x_2},\mathbf{y_2},\mathbf{z_2}).
	  \]
	  Now we observe that $\mathbf{x}=a_1\mathbf{x_1}+a_2\mathbf{x_2}$.
	  Since $\mathbf{x_1}$ and $\mathbf{x_2}$ are linearly independent, we can invert the transformation
	  and get that $a_1=L_1(\mathbf{x})$ and $a_2=L_1'(\mathbf{x})$, say, where
	  $L_1$ and $L_1'$ are linear forms with coefficients in $\setQ$, depending on $\mathbf{x_1}$
	  and $\mathbf{x_2}$. We find similar expressions for $b_1,b_2$ and $c_1,c_2$ to deduce that
	  \[
	  	 f(\mathbf{x},\mathbf{y},\mathbf{z})=L_1(\mathbf{x})L_2(\mathbf{y})L_3(\mathbf{z})+
	  	 L_1'(\mathbf{x})L_2'(\mathbf{y})L_3'(\mathbf{z}).
	  \]
	  From this expression we can see directly that if $f$ has a singular point $(X,Y,Z)$ then it must 
	  have a linear factor, without loss of generality, $L(\mathbf{x})$, say.
	  This linear factor must vanish at a point $\mathbf{p}\in\setP$.
	  We would then have for all $\mathbf{y},\mathbf{z}\in\setP$ that
	  \[
	  	0=f(\mathbf{p},\mathbf{y},\mathbf{z})=\mathbf{y}^TM_{yz}(\mathbf{p})\mathbf{z}.
	  \]
	  This then gives that $M_{yz}(\mathbf{p})=0$ which implies as above that $D=0$.
	 \item[\ref{it:item1} $\Rightarrow$ \ref{it:item2}]
	 Conversely, assume that $D=0$. By Lemma \ref{lem:DQ0}, 
	 $\Delta_{xy}(\mathbf{z})=cL(\mathbf{z})^2$. We first consider the case when $c=0$.
	 In this case $f$ factorizes by Lemma \ref{lem:factorization}. 
	 We assume without loss of generality that $f$ splits as in \eqref{eq:eqfsplitstoLB}.
	 But in this case the gradient of $f$ is zero when $(\mathbf{x},\mathbf{y},\mathbf{z})$ is
	 picked such that $L(\mathbf{y})$ and $B(\mathbf{x},\mathbf{z})$ simultaneously vanish.	 
	 By this argument, we may assume that neither of the three
	 determinants $\Delta_{xy}(\mathbf{z})$, $\Delta_{yz}(\mathbf{x})$,  or
	 $\Delta_{zx}(\mathbf{y})$ vanishes identically. Next, we pick a primitive $\mathbf{z}$ such that
	 $\Delta_{xy}(\mathbf{z})=0$. Similarly, we pick primitive $\mathbf{x}$ and $\mathbf{y}$
	 such that  $\Delta_{yz}(\mathbf{x})=0$ and $\Delta_{zx}(\mathbf{y})=0$.
	 This implies that there are primitive vectors $\mathbf{x_1}$, $\mathbf{x_2}$,
	 $\mathbf{y_1}$, $\mathbf{y_2}$, and $\mathbf{z_1}$, $\mathbf{z_2}$ such that
	 \begin{align*}
	 	\mathbf{x_1}^TM_{xy}(\mathbf{z}) &= M_{xy}(\mathbf{z})\mathbf{y_1} =0, \\
	 	\mathbf{y_2}^TM_{yz}(\mathbf{x}) &= M_{yz}(\mathbf{x})\mathbf{z_1} =0, \\
	 	\mathbf{z_2}^TM_{zx}(\mathbf{y}) &= M_{zx}(\mathbf{y})\mathbf{x_2} =0.
	 \end{align*}
	 Note that $\mathbf{x_1}^TM_{xy}(\mathbf{z})=\mathbf{z}^TM_{zy}(\mathbf{x_1})=0$.
	 Thus, $\Delta_{yz}(\mathbf{x_1})=0$. We recall that the discriminant $\Delta_{yz}$ is $D=0$
	 and that $\Delta_{yz}$ does not vanish identically. Therefore,
	 $\mathbf{x_1}=\mathbf{x}$. By using the same idea, we can show
	 \[
	 	\mathbf{x_1}=\mathbf{x_2}=\mathbf{x},\quad \mathbf{y_1}=\mathbf{y_2}=\mathbf{y},\quad
	 	\mathbf{z_1}=\mathbf{z_2}=\mathbf{z},
	 \]
	 which proves that all partial derivatives of $f$ are simultaneously zero at the non-trivial
	 point $(\mathbf{x},\mathbf{y},\mathbf{z})$. 
	 \item[\ref{it:item2} $\Leftrightarrow$ \ref{it:item3}]
	 The fact that \ref{it:item2} $\Rightarrow$ \ref{it:item3} follows
	 from Euler's theorem for homogeneous polynomials. It states in particular that for a homogeneous
	 function $g(x_1,\ldots,x_d)$ of order $n$:
	 \[
	 	\sum_{i=1}^d x_i \frac{\partial g}{\partial x_i}= ng(\mathbf{x}).
	 \]
	 In our case, we may apply this result with $g=f$, $d=6$ and $n=3$.
	 It is then clear that \ref{it:item2} $\Rightarrow$ \ref{it:item3}. The converse is
	 trivial.
\end{itemize}
\end{proof}

Let $X_1,X_2,Y_1,Y_2,Z_1,Z_2$ be real numbers $\geq 1$. Our goal is to find an upper bound
for the quantity
\[
\begin{split}
	\mathcal{N}'(\mathbf{X},\mathbf{Y},\mathbf{Z})=
	\#\{(\mathbf{x},\mathbf{y},\mathbf{z}): &|x_i|\leq X_i,
	|y_i|\leq Y_i, |z_i|\leq Z_i,\\ 
	& (x_1,x_2)=(y_1,y_2)=(z_1,z_2)=1, 
	f(\mathbf{x},\mathbf{y},\mathbf{z})=0\}.
\end{split}.
\]
By Theorem \ref{thm:thmbilinear}, it suffices to study the case when $f$ is irreducible.
We note that if $f$ factorizes then $D=0$ but the converse may not
necessarily be true as the family of examples
\[
	f(\mathbf{x},\mathbf{y},\mathbf{z})=y_{{1}}x_{{1}}z_{{1}}a_{{111}}+y_{{1}}x_{{2}}z_{{1}}a_{{211}}+y_{{1}}x
_{{2}}z_{{2}}+y_{{2}}x_{{2}}z_{{1}}
\]
shows. Here, $D=0$ but $f$ is irreducible if $a_{111}\neq 0$. 
Let 
\[
	s=s(\mathbf{x},\mathbf{y},\mathbf{z})=
	\Delta_{xy}(\mathbf{z})\Delta_{yz}(\mathbf{x})\Delta_{zx}(\mathbf{y}).
\]
It turns out that the number of points counted by $\mathcal{N}'$ for which $s=0$
may be large, even if $D\neq 0$. For example, consider
\[
	f(\mathbf{x},\mathbf{y},\mathbf{z})=y_{{1}}x_{{1}}z_{{1}}a_{{111}}+y_{{1}}x_{{2}}z_{{2}}+y_{{2}}x_{{2}}z_{
{1}}+y_{{2}}x_{{2}}z_{{2}}
\]
for $a_{111}\neq 0$. Then, $f$ is irreducible, $D=a_{111}^2\neq 0$ and
\[
s(\mathbf{x},\mathbf{y},\mathbf{z})=
a_{111}^2x_2 y_1 z_1 \left( z_{{1}}+z_{{2}} \right) \left( y
_{{1}}+y_{{2}} \right) \left( a_{{111}}x_{{1}} -x_{{2}}
 \right).
\]
Furthermore, it is easy to see that the number of points contributing to
$\mathcal{N}'$ for which $s=0$ is $\gg X_1X_2+Y_1Y_2+Z_1Z_2$. Lemma \ref{lem:lemSingPts}
shows that this is in general also the best lower bound. We therefore exclude points
for which $s=0$ in the remaining argument.

We now state the main theorem of this section:
\begin{theorem}
Let $A=(a_{ijk})$ be a $2\times 2\times 2$ matrix with associated trilinear form $f=f_A$
and hyperdeterminant $D:=\det(A)$. We assume that $f$ is irreducible over $\setZ$ and that $D\neq 0$.
Let $X_1,X_2,Y_1,Y_2,Z_1,Z_2$ be real numbers $\geq 1$ and define
\[
\begin{split}
	\mathcal{N}(\mathbf{X},\mathbf{Y},\mathbf{Z})=
	\#\{&(\mathbf{x},\mathbf{y},\mathbf{z}): |x_i|\leq X_i,
	|y_i|\leq Y_i, |z_i|\leq Z_i,\\ 
	& (x_1,x_2)=(y_1,y_2)=(z_1,z_2)=1, 
	f(\mathbf{x},\mathbf{y},\mathbf{z})=0, s(\mathbf{x},\mathbf{y},\mathbf{z})\neq 0 \}.
\end{split}
\]
Let $T=\Vert f\Vert X_1X_2Y_1Y_2Z_1Z_2$. Then
\[
	\mathcal{N}(\mathbf{X},\mathbf{Y},\mathbf{Z})\ll
	T^\epsilon\left(\frac{\sqrt{X_1X_2Y_1Y_2Z_1Z_2}}{D^{1/4}}+\sqrt{X_1X_2Y_1Y_2}+Z_1Z_2\right).
\]
\end{theorem}

We note in particular, that the condition $D\geq 1$ yields the estimate
\[
\mathcal{N}(\mathbf{X},\mathbf{Y},\mathbf{Z})\ll
	T^\epsilon\sqrt{X_1X_2Y_1Y_2Z_1Z_2}.
\]
This can easily be obtained by permuting $\mathbf{x}$, $\mathbf{y}$ and $\mathbf{z}$ in the theorem and then taking the minimum of the resulting estimates.

We define for any primitive vector $\mathbf{z}$:
\[
	f_\mathbf{z}(\mathbf{x},\mathbf{y}):=f(\mathbf{x},\mathbf{y},\mathbf{z}).
\]
With this notation, we think of $f_\mathbf{z}$ as a bilinear form in the variables $\mathbf{x},\mathbf{y}$. Let the coefficients of $f_\mathbf{z}$ be the linear forms $L_z^{(i)}=L_z^{(i)}(\mathbf{z})$, where
$i=1,\ldots,8$. We may think of $L_z^{(i)}$ as a 1x2 row vector,
the elements of $L_z^{(i)}$ being the coefficients of the form $L_z^{(i)}(\mathbf{z})$.
Let $i,j\in\{1,\ldots,8\}$. We consider a $2\times 2$ matrix which has $L_z^{(i)}$ as its first row
and $L_z^{(j)}$ as its second row. Let $D_{z}^{(i,j)}$ be the absolute value of the determinant
of this matrix.
We define analogously $f_\mathbf{y},L_y^{(i)}, D_{y}^{(i,j)}$ and
$f_\mathbf{x},L_x^{(i)}, D_{x}^{(i,j)}$.

\begin{lemma}\label{lem:q2dD}
	Let $q$ be a non-zero integer. Assume that there exists a primitive vector $\mathbf{z}$
	such that $q\divides L_z^{(i)}(\mathbf{z})$ for all $i$. Then $q^2\divides D$.
\end{lemma}
\begin{proof}
Let
\[
	g_z=\gcd_{i,j} D_z^{(i,j)}.
\]
Our first aim is to prove that $q\divides g_z$. We fix $i$ and $j$ and want to show that
$q\divides D_z^{(i,j)}$. If $D_z^{(i,j)}=0$ then we are done. Thus, we assume that $D_z^{(i,j)}\neq 0$.
Let $L_z^{(i)}(\mathbf{z})=az_1+bz_2$ and $L_z^{(j)}(\mathbf{z})=cz_1+dz_2$ so that
$D_z^{(i,j)}=ad-bc$. There exists integers $k_1$ and $k_2$ such that
\begin{align*}
	az_1+bz_2=k_1q\quad \text{ and }\quad cz_1+dz_2=k_2q.
\end{align*}
So that
\begin{align*}
	(bc-ad)z_2 &=q(ck_1-ak_2),\\
	(bc-ad)z_1 &=-q(dk_1-bk_2)
\end{align*}
We observe that $q\divides (bc-ad)z_1$ and that $q\divides (bc-ad)z_2$.
Since $z_1$ and $z_2$ are coprime, we therefore get that $q\divides (ad-bc)(z_1,z_2)=ad-bc$. 
This proves that $q\divides D_z^{(i,j)}$ and since $i$ and $j$ were arbitrary,
we may deduce that $q\divides g_z$.
It can be shown by direct calculations that $g_z$ is the greatest common factor of the following
six 2x2 determinants:
\begin{alignat*}{3}
	&a_{111} a_{122} - a_{112} a_{121}, && \qquad
     a_{111} a_{212} - a_{112} a_{211}, && \qquad
     a_{111} a_{222} - a_{112} a_{221},\\
    &a_{212} a_{121} - a_{211} a_{122}, && \qquad
     a_{121} a_{222} - a_{122} a_{221}, && \qquad
     a_{211} a_{222} - a_{212} a_{221}
\end{alignat*}
From this it follows at once that $q^2\divides D$.
\end{proof}

\begin{lemma}\label{lem:allLizero}
	Assume that there exists a primitive vector $\mathbf{z}$
	such that $L_z^{(i)}(\mathbf{z})=0$ for all $i$. Then $f$ factorizes or vanishes identically.
\end{lemma}
\begin{proof}
Let $L_z^{(i)}(\mathbf{z})=a_iz_1+b_iz_2$ ($i=1,\ldots,8$), say and define
$a_i'=a_i/(a_i,b_i)$ and $b_i'=b_i/(a_i,b_i)$ in the case when $(a_i,b_i)\neq 0$.
We recall that $z_1$ and $z_2$ are co-prime.
We assume that all $L_z^{(i)}$ vanish simultaneously for the same $\mathbf{z}$.
If $z_1=0$ then $z_2=\pm 1$ and therefore $b_i'=0$ and $a_i'=\pm 1$ for all $i$.
This implies that $b_i=0$ for all $i$, which means that $f$ has a linear factor $z_1$ and the 
conclusion in the lemma is valid. Thus we may assume that $z_1\neq 0$ and similarly $z_2\neq 0$.
If $z_1z_2\neq 0$ and
$L_z^{(i)}(\mathbf{z})=0$ then either $a_i=b_i=0$ or $\mathbf{z}=\pm (b_i',-a_i')$.
This implies that either all $L_z^{(i)}$ are identically $0$ or that all the $L_z^{(i)}$ are pairwise proportional. In the first case, $f$ vanishes identically and in the second case, $f$ has a linear factor. This proves the lemma.
\end{proof}

\begin{lemma}\label{lem:lemSingPts}
	Assuming that $f$ is irreducible over $\setZ$, the number of points 
	$(\mathbf{x},\mathbf{y},\mathbf{z})$ counted by $\mathcal{N}$
	for which $\Delta_{xy}(\mathbf{z})=0$ is $O(X_1 X_2+Y_1 Y_2)$.
\end{lemma}
\begin{proof}
We first note that $\Delta_{xy}$ cannot vanish identically by Lemma \ref{lem:factorization}.
Thus, there are $O(1)$ choices for $\mathbf{z}$. We fix one such $\mathbf{z}$ and
consider the equation $f_{\mathbf{z}}(\mathbf{x},\mathbf{y})=0$ as a bilinear equation
with coefficients $L_z^{(i)}(\mathbf{z})=a_iz_1+b_iz_2$ ($i=1,\ldots,8$), say.
By Lemma \ref{lem:allLizero}, the function $f_{\mathbf{z}}(\mathbf{x},\mathbf{y})$ cannot vanish
identically. The equation $f_{\mathbf{z}}(\mathbf{x},\mathbf{y})=0$ is therefore saying that a non-zero bilinear form vanishes.
This bilinear form will factorize as a product of linear forms
$f_{\mathbf{z}}(\mathbf{x},\mathbf{y})=L_{1,\mathbf{z}}(\mathbf{x}) L_{2,\mathbf{z}}(\mathbf{y})$
because $\Delta_{xy}(\mathbf{z})=0$. It is clear that the equation
$L_{1,\mathbf{z}}(\mathbf{x}) L_{2,\mathbf{z}}(\mathbf{y})=0$ has $\ll X_1X_2+Y_1Y_2$ solutions.
\end{proof}

Fix a primitive integer vector $\mathbf{z}$.
We want to apply Theorem \ref{thm:thmbilinear} for bilinear equations to $f_\mathbf{z}$. 
By assumption, $\Delta(f_\mathbf{z})\neq 0$.
If there exists an integer $q$ dividing all $L_z^{(i)}$ then $q^2\divides D$ by Lemma \ref{lem:q2dD}.
Let $L_i'=L_z^{(i)}/q$. We divide the equation $f_\mathbf{z}(\mathbf{x},\mathbf{y})=0$
by $q$ so that we get $f'(\mathbf{x},\mathbf{y})=0$,
where $f'$ is a bilinear form with coprime coefficients and non-zero determinant 
$\Delta(f')=\Delta_{xy}(\mathbf{z})/q^2$. Thus, we may deduce that:
\[
	\mathcal{N}:\ll \sum_{q^2\divides D}\sum_{\substack{\mathbf{z}: |z_i|\leq Z_i\\
											  (z_1,z_2)=1\\
											  q\divides L_z^{(j)}(\mathbf{z})\\
											 |\Delta_{xy}(\mathbf{z})|\geq 1}}
		\dfunc(\Delta_{xy}(\mathbf{z}))
	\left(\frac{q\sqrt{X_1X_2Y_1Y_2}}{\sqrt{|\Delta_{xy}(\mathbf{z})|}}+1\right).
\]
Next, we process the condition that 
$q\divides L_z^{(j)}(\mathbf{z})$ for all $j$.
Let $C$ be the $4\times 2$ matrix having the $L_z^{(j)}$ as rows.
Since $f$ is irreducible, there exists no prime $p$ that divides all elements of $C$.
Since $q\divides L_z^{(j)}(\mathbf{z})$, we have as in the proof of
of Lemma \ref{lem:q2dD} that $q$ divides all 2x2 minors of $C$. Therefore, we may apply Lemma
\ref{lem:lattice} and deduce that the elements $\mathbf{z}$ counted by the above inner sum are in a
lattice $\Lambda_{q}$ of determinant $q$. We further observe that
the points $\mathbf{z}$ counted by the inner sum are also in the ellipse $E(Z_1,Z_2)$ given by
$(z_1/Z_1)^2+(z_2/Z_2)^2\ll 1$. This ellipse has area $A(E(Z_1,Z_2))\asymp Z_1Z_2$. 
By Lemma \ref{lem:lemEllipse}, there exists a basis $\mathbf{b}_1$,
$\mathbf{b}_2$ of the lattice $\Lambda_q$ and a positive number $\alpha$ 
such that if we write $\mathbf{z}\in E(Z_1,Z_2)$
as $\mathbf{z}=G\mathbf{z}'$ then $|z_1'|\leq\alpha$ and $|z_2'|\ll Z_1Z_2/(\alpha q)$.
Here, $G$ is a $2\times 2$ matrix with columns $\mathbf{b}_1$ and $\mathbf{b}_2$.
Note that $\Delta_{xy}$ is a quadratic form of discriminant $D$ and thus, the discriminant of the
new quadratic form $\Delta'_{xy}(\mathbf{z}')=\Delta_{xy}(G\mathbf{z}')$ is $q^2 D$. We further note that if $\Delta_{xy}'(\mathbf{z'})=0$ then $\Delta_{xy}(\mathbf{z})=0$, and thus it is safe to assume that 
$|\Delta_{xy}'(\mathbf{z'})|\geq 1$. By following the proof of Lemma \ref{lem:lemEllipse},
we can see that 
\[
	\Vert\Delta_{xy}'\Vert\ll\Vert\Delta_{xy}\Vert P(D,Z_1,Z_2), 
\]
where $P(D, Z_1,Z_2)$ is a finite power of $D Z_1Z_2$. Thus, after a change of variables
$\mathbf{z}=G_q\mathbf{z}'$ for each $q$, we arrive at the estimate
\[
	\mathcal{N}\ll \sum_{q^2\divides D}\sum_{\substack{\mathbf{z'}: |z_i'|\leq Z_i'\\
											  (z_1',z_2')=1\\
											 |\Delta_{xy}'(\mathbf{z'})|\geq 1}}
						d(\Delta_{xy}'(\mathbf{z}'))
	\left(\frac{q\sqrt{X_1X_2Y_1Y_2}}{\sqrt{|\Delta_{xy}'(\mathbf{z'})|}}+1\right),
\]
where $Z_1'Z_2'\ll Z_1Z_2/q$. If $Z_1'<1$ then
$z_1'=0$ and $z_2'=\pm 1$. A similar argument holds if $Z_2'<1$. Thus, we note that
\[
	1\leq |\Delta_{xy}'(\mathbf{z'})|\ll \Vert\Delta_{xy}\Vert P(D,Z_1,Z_2),
\]
where again $P(D,Z_1,Z_2)$ denotes some finite power of $DZ_1Z_2$. We may therefore deduce the 
trivial estimate
\[
	\sum_{q^2\divides D}\sum_{\substack{\mathbf{z'}: |z_i'|\leq Z_i'\\
											  (z_1',z_2')=1\\
											 |\Delta_{xy}'(\mathbf{z'})|\geq 1}}
	d(\Delta_{xy}'(\mathbf{z}'))\ll T^\epsilon Z_1 Z_2
\]
and it remains to find an upper bound for the remaining sum. Thus, we may split the
sum over $\mathbf{z}'$ into dyadic ranges for $|\Delta_{xy}'(\mathbf{z'})|$. In particular, there exists
an integer $R$ such that $1\leq R\ll \Vert\Delta_{xy}\Vert P(D,Z_1,Z_2)$ and 
\[
	\mathcal{N}\ll T^\epsilon \sqrt{X_1X_2Y_1Y_2}
	\sum_{q^2\divides D}q
	\sum_{\substack{\mathbf{z'}: |z_i'|\leq Z_i'\\
								(z_1',z_2')=1\\
								R\leq |\Delta_{xy}'(\mathbf{z'})|<2R}}
	\frac{1}{\sqrt{|\Delta_{xy}'(\mathbf{z'})|}}+T^\epsilon Z_1Z_2
\]
We proceed to attack the inner sum. For $q,n>0$ with $q^2\divides D$, let
\[
	a_{n,q}:=\card{\mathbf{z}'\in\setZ^2: |z_i'|\leq Z_i', (z_1',z_2')=1, |\Delta_{xy}(\mathbf{z}')|=n}.
\]
Then
\[
	\mathcal{N}\ll T^\epsilon \sqrt{X_1X_2Y_1Y_2}
	\sum_{q^2\divides D}qS_q(R)+T^\epsilon Z_1 Z_2,
\]
where 
\[
	S_q(R)=\sum_{R\leq n<2R}\frac{a_{n,q}}{\sqrt{n}}.
\]
We now need to find an upper bound for $S_q(R)$. For $t\geq 1$, let
\[
	A_q(t):=\sum_{1\leq m<2t}a_{m,q}.
\]
By partial summation, we obtain
\[
	S_q(R)\ll \frac{A_q(2R)}{\sqrt{R}}+\frac{A_q(R)}{\sqrt{R}}+\int_R^{2R}\frac{A_q(t)}{t^{3/2}}\,dt.
\]
It remains to find an upper bound for $A_q(t)$. We prove the following lemma:
\begin{lemma}
The following upper bound holds:
\[
	A_q(t)\ll\min\left\{ \frac{Z_1Z_2}{q}+1, \left(\frac{t}{qD^{1/2}}+1\right) (Tt)^\epsilon \right\}
\]	
\end{lemma}
\begin{proof}
	Note that
	\[
		A_q(t)=\card{\mathbf{z}\in\setZ^2: |z_i|\leq Z_i', (z_1,z_2)=1,
		1\leq|\Delta_{xy}'(\mathbf{z})|\leq t},
	\]
	where we recall that $Z_1'$ and $Z_2'$ are positive numbers such that $Z_1'Z_2'\ll Z_1Z_2/q$ and
	$\Delta_{xy}'$ is a quadratic form with integer coefficients such that $D(\Delta_{xy}')=q^2D$.
	If $Z_1'<1$ then $z_1=0$ and $z_2=\pm 1$. Thus, the case when $Z_1'<1$ or $Z_2'<1$
	contributes $O(1)$ to $A_q(t)$. If $Z_1'\geq 1$ and $Z_2'\geq 1$ then
	\[
		A_q(t)\ll (Z_1'+1)(Z_2'+1)\ll \frac{Z_1Z_2}{q}+1.
	\]
	This proves the first bound.
	Next, let us consider the case when 
	\[
		\Delta_{xy}'(\mathbf{z})=\alpha (z_1-\beta_1 z_2) (z_1-\beta_2 z_2)
								=\alpha L_1(\mathbf{z})L_2(\mathbf{z}),
	\]
	say, for some $\alpha\in\setZ\setminus\{0\}$ and $\beta_i\in\setC$. Note that in general
	$D(\Delta_{xy}')=|\alpha|^2|\beta_1-\beta_2|^2$. Such a factorization exists if 
	and only if the coefficient of $z_1^2$ in $\Delta'$ is non-zero.
	We first consider the case when $\beta_1=\beta_2$. In this case we can deduce that 
	$q^2D=D(\Delta_{xy}')=0$, which is impossible since $D\neq 0$.\bigskip\\ 
	Next consider the case when $\beta_1=\bar{\beta_2}=a+b i$, with $a,b\in\setR$ and $b\neq 0$. Then
	\[
		|\Delta_{xy}'(\mathbf{z})|=|\alpha| ((x-ay)^2+(by)^2).
	\]
	We set $r=x-ay$ and $s=by$. Then the point $(r,s)$ is inside a real lattice $\Lambda$
	given by the matrix
	\[
		\left(\begin{array}{cc}
		1 & -a \\ 0 & b \\	
		\end{array} \right).
	\] 
	We note that $\det(\Lambda)=|b|$ and that the point $(r,s)$ also lies in the circle given by
	\[
		r^2+s^2\leq t/|\alpha|.
	\]
	This circle has area $\pi t/|\alpha|$. Thus, by Lemma \ref{lem:lemEllipse}, the number of
	possibilities for $(r,s)$ is
	\[
		\ll \frac{t}{|\alpha||b|}+1\ll \frac{t}{qD^{1/2}}+1,
	\]
	where in the last line we used that $q^2D=D(\Delta_{xy}')=4b^2\alpha^2$.
	Each choice of $(r,s)$ gives at most one choice for $\mathbf{z}$ since $b\neq 0$.\bigskip\\
	We may therefore assume that $\beta_1$ and $\beta_2$ are real.
	By the solution formula
	for quadratic polynomials, we observe that $\beta_i\ll \Vert\Delta'\Vert/|\alpha|$ 
	and therefore,
	$\alpha L_i(\mathbf{z})\ll \Vert\Delta'\Vert Z$, where $Z=\max\{Z_1',Z_2'\}$. 
	Since $1\leq |\Delta_{xy}'(\mathbf{z'})|$,
	we also note that $|L_i(\mathbf{z})|\gg 1/(\Vert\Delta'\Vert Z)$. We therefore conclude that
	\[
		\frac{1}{\Vert\Delta'\Vert Z}\ll |L_i(\mathbf{z})|\ll\Vert\Delta'\Vert Z.
	\]
	We proceed by splitting the ranges of $L_1(\mathbf{z})$ and $L_2(\mathbf{z})$ into dyadic ranges
	so that 
	\[
		K\leq |L_1(\mathbf{z})|<2K\text{ and } L\leq |L_2(\mathbf{z})|<2L
	\]
	for some integers $K$ and $L$. The dyadic subdivision comes at the cost of a factor
	$(\Vert\Delta_{xy}\Vert D Z_1Z_2 t)^\epsilon$ in the estimate for $A_q(t)$
	By setting $r=L_1(\mathbf{z})$ and $s=L_2(\mathbf{z})$ we can again see that the points $(r,s)$
	are in a real lattice $\Lambda$ with $|\det(\Lambda)|=|\beta_1-\beta_2|$ and they are in an ellipse
	of shape
	\[
		(r/K)^2+(s/L)^2\ll 1
	\]
	and area $\ll KL\ll \frac{t}{|\alpha|}$. And thus, we obtain once again by Lemma
	\ref{lem:lemEllipse} that the number of possibilities for the points $(r,s)$
	is 
	\[
		\ll (\Vert\Delta_{xy}\Vert D Z_1Z_2 t)^\epsilon (t/(qD^{1/2})+1).
	\]
	The linear transformation defining $(r,s)$ is invertible
	because $\beta_1\neq \beta_2$. Thus, for each choice $(r,s)$ there is again at 
	most one possible value for $\mathbf{z}$.
	This finishes the proof of the lemma if the coefficient
	of $z_1^2$ in $\Delta'$ is non-zero. The case when the coefficient of $z_2^2$ is non-zero
	is similar.\bigskip\\
	Therefore, we may now concentrate on the last case when 
	$\Delta_{xy}'(\mathbf{z})=b z_1 z_2$ for some integer $b\neq 0$.
	We recall that $D(\Delta_{xy}')=q^2D=b^2$. In particular, $b=qD^{1/2}$.
	We observe that 
	\[
		A_q(t)\ll \card{\mathbf{z}\in\setZ^2: 1\leq |b z_1 z_2|\leq t}
		\ll t^\epsilon\left(\frac{t}{b}+1\right)
		\ll t^\epsilon\left(\frac{t}{qD^{1/2}}+1\right).
	\]
\end{proof}
We note that $q<\sqrt{R}$ because $q^2\divides\Delta_{xy}\ll R$.
This gives
\[
	S_q(R)\ll (TR)^\epsilon \left(
	\frac{1}{q}\min\left\{ \frac{Z_1Z_2}{\sqrt{R}}, \frac{\sqrt{R}}{D^{1/2}} \right\}
	+\frac{1}{\sqrt{R}}\right)
	\ll \frac{T^\epsilon}{q}\left(\frac{\sqrt{Z_1 Z_2}}{D^{1/4}}+1\right),
\]
where the critical value is obtained when $R=Z_1Z_2 D^{1/2}$. This proves the theorem.

\addcontentsline{toc}{section}{Bibliography}
\bibliographystyle{plain}

\end{document}